\documentclass[11pt, a4paper]{amsart}
\usepackage{amsmath}%
\usepackage{amsfonts}%
\usepackage{amssymb}%
\usepackage{graphicx}
\usepackage{amsthm}
\usepackage[foot]{amsaddr}
\usepackage{amssymb}
\usepackage{tikz}
\usepackage{tikz-cd}
\usetikzlibrary{matrix,arrows}
\usepackage[english, ngerman]{babel}
\usepackage[utf8]{inputenc}
\usepackage{hyperref}
\usepackage{cite}
\usepackage{footnote}
\usepackage{stmaryrd}
\SetSymbolFont{stmry}{bold}{U}{stmry}{m}{n}
\usepackage[a4paper,margin=3cm]{geometry}
\usepackage{verbatim}
\usepackage{enumerate}

\newtheorem{thm}{Theorem}
\theoremstyle{plain}

\newtheorem{cor}[thm]{Corollary}

\newtheorem{prop}[thm]{Proposition}

\theoremstyle{remark}

\newtheorem*{acknowledgements}{Acknowledgements}

\newcommand{\BH}{{\mathbb{H}}}

\newcommand{\BN}{{\mathbb{N}}}

\newcommand{\BR}{{\mathbb{R}}}

\newcommand{\BZ}{{\mathbb{Z}}}

\newcommand{\CB}{{\mathcal B}}

\newcommand{\CI}{{\mathcal I}}

\newcommand{\CK}{{\mathcal K}}

\newcommand{\vol}{\mathop{\rm vol}\nolimits}

\newcommand{\vect}[1]{{\boldsymbol{#1}}}

\renewcommand{\mod}{\mathop{\rm mod}\nolimits}

\newcommand{\SL}[1]{\mathop{\rm SL}_{#1} \nolimits}

\newcommand{\tr}{\mathop{\rm tr}\nolimits}
\newcommand{\nr}{\mathop{\rm nr}\nolimits}

\newcommand{\quotient}[2]{
        \mathchoice
            {
                \text{\raise1ex\hbox{$#1$}\Big/\lower1ex\hbox{$#2$}}%
            }
            {
                #1\,/\,#2
            }
            {
                #1\,/\,#2
            }
            {
                #1\,/\,#2
            }
    }
		
\newcommand{\lquotient}[2]{
        \mathchoice
            {
                \text{\lower1ex\hbox{$#1$}\Big \backslash \raise01ex\hbox{$#2$}}%
            }
            {
                #1\,\backslash\,#2
            }
            {
                #1\,\backslash\,#2
            }
            {
                #1\,\backslash\,#2
            }
    }

\newcommand{\rquotient}[2]{
        \mathchoice
            {
                \text{\raise01ex\hbox{$#1$}\Big/\lower1ex\hbox{$#2$}}%
            }
						{
                #1\,/\,#2
            }
            {
                #1\,/\,#2
            }
            {
                #1\,/\,#2
            }
    }
		
\newcommand{\lrquotient}[3]{
        \mathchoice
            {
                \text{\lower1ex\hbox{$#1$}\Big \backslash \raise01ex\hbox{$#2$}\Big/\lower1ex\hbox{$#3$}}%
            }
            {
                #1\,\backslash\,#2\,/\,#3
            }
            {
                #1\,\backslash\,#2\,/\,#3
            }
            {
                #1\,\backslash\,#2\,/\,#3
            }
    }

\newcommand{\sm}{\left(\begin{smallmatrix}}
\newcommand{\esm}{\end{smallmatrix}\right)}
\newcommand{\bpm}{\begin{pmatrix}}
\newcommand{\ebpm}{\end{pmatrix}}

\newcommand{\primesum}{\sideset{}{'}\sum}

\numberwithin{equation}{section}
\begin{document}
\selectlanguage{english}

\bibliographystyle{plain}

\title{Sup-norm of Hecke--Laplace eigenforms on $S^3$}
\author{Raphael S. Steiner}
\address{Institute for Advanced Study -- FH 317, 1 Einstein Drive, Princeton NJ 08540, USA}%
\email{raphael.steiner.academic@gmail.com}%


\subjclass[2010]{11F72 (11F27, 58G25) }
\keywords{Sup-norm, Hecke, Laplace, eigenforms, arithmetic manifold, fourth moment}


\begin{abstract} We prove sub-convex bounds on the fourth moment of Hecke--Laplace eigenforms on $S^3$. As a corollary, we get a Burgess-type sub-convex bound on the sup-norm of an individual eigenform. This constitutes an improvement over what is achievable through employing the Iwaniec--Sarnak amplifier.
\end{abstract}
\maketitle


\section{Introduction}

Let $X$ be a compact Riemannian manifold and denote by $\Delta$ the corresponding Laplace--Beltrami operator acting on functions on $X$. Given an eigenfunction $\phi$ of $\Delta$ with eigenvalue $-\lambda$ (i.e. $\Delta \phi + \lambda \phi=0$), it is a classical question to bound the sup-norm of $\phi$. In general (see for example \cite{SeegerSogge}), one has the bound
\begin{equation}
\|\phi\|_{\infty} \le C_X  \cdot (1+\lambda)^{\frac{\dim X-1}{4}} \|\phi\|_2,
\label{eq:convexitybound}
\end{equation}
for some constant $C_X$, which depends only on $X$. In this generality, one cannot do better as equality is attained for the spheres $S^n$. The main obstruction being that on the spheres the multiplicity $m_{\lambda}$ of an eigenvalue $\lambda$ is large. In fact, as large as $\lambda^{\frac{n-1}{2}}$, which combined with the lower bound (see for example \cite{SarnakLetterMora})

$$
m_{\lambda} \le \vol(X) \sup_{\substack{\Delta \phi + \lambda \phi =0}} \frac{\|\phi\|_{\infty}^2}{\|\phi\|_2^2},
$$
shows that \eqref{eq:convexitybound} is sharp. In contrast, for negatively curved Riemannian surfaces the multiplicities are expected to be small and it is conjectured in \cite{ArithmeticQC} that there one should have $\| \phi \|_{\infty} \ll_{\epsilon} (1+\lambda)^{\epsilon}$. Here and hereafter, we shall adopt Vinogradov's natation, i.e. an expression $F \ll_{\epsilon} G(\epsilon)$ shall mean for every sufficiently small $\epsilon>0$ we have $|F|\le C_{\epsilon} G(\epsilon)$ for some constant $C_{\epsilon}$, which may depend on $\epsilon$, on the whole domain on which the expression makes sense unless stated otherwise. However, even in this case, not much beyond the bound in \eqref{eq:convexitybound} is known. Only an extra factor of $\log(2+\lambda)$ has been saved over \eqref{eq:convexitybound}, i.e. the sup-norm is bounded by $(1+\lambda)^{\frac{1}{4}}/\log(2+\lambda)$. In a breakthrough paper \cite{IS95}, Iwaniec and Sarnak have demonstrated a new method to bound the sup-norm of certain arithmetic surfaces of negative curvature. They achieved the bound 
\begin{equation}
\|\phi\|_\infty \ll_{\epsilon} (1+\lambda)^{\frac{5}{24}+\epsilon}
\label{eq:ISbound}
\end{equation}
for Hecke--Laplace eigenforms $\phi$. This constitutes a power saving over the bound \eqref{eq:convexitybound}. Their method has been adopted by many in numerous other contexts. In particular, in relation to our result, we shall mention Vanderkam \cite{Vanderkam}, who extended their argument to the positively curved surface $S^2$, and Blomer--Michel \cite{BM2011, BM2013}, who not only considered the eigenvalue, but also the volume aspect of the sup-norm of certain arithmetic $d$-fold copies of $S^2$ and $S^3$, however they left the eigenvalue aspect of $S^3$ for future work, which has not appeared so far. 

Before we state our results, we shall introduce some notation. We shall identify a point $\vect{x}=(x_1,x_2,x_3,x_4)$ on $S^3$ with the quaternion $\vect{x}=x_1+x_2i+x_3j+x_4k$. This identifies $S^3$ with $B^1(\BR)$, the subspace of the quaternions with norm $1$. Here, we denoted by $B$ the standard quaternion algebra. We denote by $\overline{\vect{x}}=x_1-x_2i-x_3j-x_4k$ the conjugate of $\vect{x}$ and by $\tr(\vect{x})=\vect{x}+\overline{\vect{x}}$, $\nr(\vect{x})=\vect{x}\overline{\vect{x}}=\overline{\vect{x}}\vect{x}$ the reduced trace and reduced norm of $\vect{x}$, respectively. As in \cite{HeckeOpS2II}, we may define Hecke operators on $L^2(S^3)$ by
$$
(T_Nf)(\vect{x}) = \frac{1}{8} \sum_{\substack{\vect{m} \in B(\BZ) \\ \nr(\vect{m})=N}} f\left( \frac{\vect{m}}{\sqrt{N}} \vect{x} \right).
$$
These Hecke operators are self-adjoint, commute with each other, and commute with the Laplace operator and thus can be simultaneously diagonalised. Furthermore, the Hecke operators are multiplicative, i.e. they satisfy $T_M \circ T_N=T_{MN}$ for $(M,N)=1$, and satisfy $T_{p^{\alpha+1}}=T_{p^{\alpha}}\circ T_p-pT_{p^{\alpha-1}}$ for $p$ an odd prime and $\alpha\in \BN$. We call an eigenform of all the Hecke operators as well as the Laplace operator a Hecke--Laplace eigenform. 
For a set of eigenvalues $\vect{\lambda}=\{\lambda(1),\lambda(3),\lambda(5),\dots\}$ of the odd Hecke operators, we denote by $V_{\vect{\lambda}}$ the common eigenspace, i.e. $\phi \in V_{\vect{\lambda}} \Leftrightarrow \forall N \in \BN, N\text{ odd}: T_N \phi = \lambda(N) \phi$. We are now able to state our theorem.

\begin{thm} Let $\{\phi_j\}$ be an orthonormal basis of Hecke--Laplace eigenforms of the $(-n(n+2))$-Laplace eigenspace on $S^3$. Then, we have
$$
\sup_{\vect{x} \in S^3} \sum_{\substack{\vect{\lambda}\\ \lambda(1)=1}} \left( \sum_{\phi_j \in V_{\vect{\lambda}}} |\phi_j(\vect{x})|^2 \right)^2  \ll_{\epsilon} n^{3+\epsilon}.
$$
\label{thm:22moment}
\end{thm}
One immediately gets the following two corollaries.
\begin{cor} Let $\{\phi_j\}$ be an orthonormal basis of Hecke--Laplace eigenforms of the $(-n(n+2))$-Laplace eigenspace on $S^3$. Then, we have
$$
\sup_{\vect{x} \in S^3} \primesum_{j} |\phi_j(\vect{x})|^4 \ll_{\epsilon} n^{3+\epsilon}.
$$
Here, the sum over $j$ is restricted to those Hecke--Laplace eigenforms with eigenvalue $1$ for the Hecke operator $T_1$.
\label{cor:4moment}
\end{cor} 
\begin{cor} For a Hecke--Laplace eigenform $\phi$ on $S^3$ with $T_1\phi=\phi$, we have
$$
\|\phi\|_{\infty} \ll_{\epsilon} (1+\lambda)^{\frac{3}{8}+\epsilon} \|\phi\|_2.
$$
\label{cor:indvsup}
\end{cor}
A few remarks are in order. Theorem \ref{thm:22moment} is sharp by Cauchy--Schwarz and the pre-trace formula \eqref{eq:pretrace} as there are about $n$ values of $\vect{\lambda}$ such that $\# \{\phi_j \in V_{\vect{\lambda}} \}>0$. Corollary \ref{cor:4moment}, although not optimal, marks a significant improvement over the trivial bound. Indeed, it marks the halfway point between the trivial bound $n^4$ and the lower bound of $n^2$. One might speculate whether the fourth moment can be as small as $n^{2+\epsilon}$. Corollary \ref{cor:indvsup} marks a Burgess-type sub-convexity bound for the sup-norm. In particular it constitutes an improvement over what is achievable through employing the amplifier of Iwaniec--Sarnak \cite{IS95} as the latter would only yield an exponent of $\frac{5}{12}$.

The idea behind Theorem \ref{thm:22moment} is rather simple. We take the pre-trace formula \eqref{eq:pretrace}, attach a theta series to both sides, and use Parseval. In order to get to the left hand side of the inequality of Theorem \ref{thm:22moment}, we invoke a lower bound of the Petersson norm of an arithmetically normalised newform due to Hoffstein--Lockhart \cite{HoffLock}. The upper bound follows from a lattice point counting argument. Our argument may, in some sense, be seen as a realisation of a remark given in \cite{IS95}. They remark that their result can be improved to the same qualitative bound as given here (Corollary \ref{cor:indvsup}) if one has a good lower bound on
$$
\sum_{m \le N} |\lambda(m)|^2.
$$
To the best of the author's knowledge, this has only been achieved for the Eisenstein series, see \cite{SupEisenstein},\cite{HuangXu}, and dihedral Maass forms \cite{supdihedral}. In our approach, we circumnavigate this issue by replacing the above sum with the residue at $s=1$ of the Rankin--Selberg convolution, which, here, is proportional to the Petersson norm of the corresponding theta series. Additionally, we shall point out a fundamental difference in the two methods. Their argument only gives a bound on a single form whereas our argument yields a bound on the fourth moment over a whole family.



There is potential to improve upon Corollary \ref{cor:4moment} by also considering the Hecke operators acting on the right. Furthermore, it appears that this idea of using the theta lift to double the moment is a fairly general one and can be used to tackle various (sup-)norm problems, subject to being able to deal with the geometric side. In upcoming work with I. Khayutin, we shall extend the technique to holomorphic forms on arithmetic hyperbolic surfaces. Additionally, as the author retrospectively has found, this idea can and has been used by Nelson \cite{NelsonQV1,NelsonQV2,NelsonQV3} as a starting point for estimating/evaluating the Quantum Variance.

\begin{acknowledgements} I would like to thank Shai Evra, Djordje Mili\'cevi\'c, and Peter Sarnak for interesting discussions on this and related topics. I would further like to thank the referees for their detailed remarks, which have improved the overall quality of the paper.

This work started during my Ph.D. at the University of Bristol and was completed at the Institute for Advanced Study in Princeton, where I am currently supported by the National Science Foundation Grant No. DMS -- 1638352 and the Giorgio and Elena Petronio Fellowship Fund II.
\end{acknowledgements}

\section{Proof}

As in the theorem, we shall denote by $\{\phi_j\}$ an orthonormal basis of Hecke--Laplace eigenforms of the $(-n(n+2))$-Laplace eigenspace, where $n \in \BN$. We have the following pre-trace formula on $S^3$
\begin{equation}
\frac{1}{n+1} \sum_j \phi_j(\vect{x}) \overline{\phi_j(\vect{y})} = U_n\left(\tfrac{1}{2}\tr (\vect{x}\overline{\vect{y}})\right),
\label{eq:pretrace}
\end{equation}
where $U_n$ is the $n$-th Chebyshev polynomial of the second kind, given by
\begin{equation}
U_n(\cos(\theta))= \frac{\sin((n+1)\theta)}{\sin(\theta)}.
\label{eq:Un}
\end{equation}
This may be easily deduced from the fact, that $\frac{1}{n+1}U_n(\cos(\theta))$ is the unique (normalised) zonal spherical function of degree $n$ (cf. \cite[Eq. (6.26)]{thesis}). 
We may restrict ourselves now to even integers $n$ as otherwise \eqref{eq:pretrace} shows that $T_1$ identically vanishes as $U_n$ is an odd function for $n$ odd. We shall now consider the following theta series for $z$ in the upper half-plane $\BH$:
\begin{align}
F_n(\vect{x},\vect{y};z) &= \sum_{\vect{m} \in B(\BZ)} \nr(\vect{m})^{\frac{n}{2}} U_n \left( \frac{1}{2 \sqrt{\nr(\vect{m})}} \tr(\vect{m}\vect{x}\overline{\vect{y}})  \right) e(\nr(\vect{m})z), \label{eq:Fn}\\
\Phi_j(z) &= \sum_{N \ge 1}  \lambda_j(N) N^{\frac{n}{2}} e\left( N z\right) \label{eq:Phij}.
\end{align}
We shall suppress the dependence on $\vect{x},\vect{y}$ in $F_n$ and write $F_n(z)=F_n(\vect{x},\vect{y};z)$ for short. It follows from \cite[Ch. 10]{IwClassic} that both $F_n$ and $\Phi_j$ are cusp forms on $\Gamma_0(4)$ of weight $n+2$ and trivial multiplier system. A bit more is true. $G_n(z)=F_n(\frac{z}{2})$ is a cusp form for the theta subgroup $\Gamma_{\theta}=\{\gamma \in \SL{2}(\BZ) |\gamma \equiv \sm 1 & 0 \\ 0 & 1 \esm\text{ or } \sm 0 & 1 \\ 1 & 0 \esm \mod(2) \}$ of weight $n+2$ and the multiplier system $\upsilon_{\theta}$ which takes on the value $1$ or $-1$ depending whether $\gamma \equiv \sm 1 & 0 \\ 0 & 1 \esm \mod(2)$ or not. For later use, we shall also require the Fourier expansion of $G_n(z)$ at the cusp $1$. It is given by
\begin{equation}
\left(G_n|_{n+2} \sm 1 & -1 \\ 1 & 0 \esm\right)(z) = - \!\! \! \!\sum_{\vect{m} \in B(\BZ)+\vect{\xi}}\nr(\vect{m})^{\frac{n}{2}} U_n \left( \frac{1}{2 \sqrt{\nr(\vect{m})}} \tr(\vect{m}\vect{x}\overline{\vect{y}})  \right) e\left(\tfrac{1}{2}\nr(\vect{m})z\right),
\label{eq:cusp1}
\end{equation}
where $\vect{\xi}=\frac{1}{2}(1+i+j+k)$. Furthermore, $\Phi_j(z)$ are Hecke eigenforms due to the multiplicative nature of its Fourier coefficients. The pre-trace formula \eqref{eq:pretrace} now implies
\begin{equation}
\frac{8}{n+1} \sum_j \phi_j(\vect{x})\overline{\phi_j(\vect{y})} \Phi_j(z) = F_n(z).
\label{eq:maineq}
\end{equation}
We set $\vect{x}=\vect{y}$ and wish to compute $\|\tfrac{n+1}{8}F_n\|^2$ in two ways. For this endeavour we require an orthonormal basis of Hecke eigenforms of the space of cusp forms on $\Gamma_0(4)$ of weight $n+2$. Such a basis has been computed by Blomer--Mili\'cevi\'c \cite{BlomONB}\footnote{Corrections can be found at http://www.uni-math.gwdg.de/blomer/corrections.pdf.}:
$$
\bigcup_{\substack{l|4}}\bigcup_{\substack{h \in \CB(l)}}\left\{ h_d(z)= \sum_{e|d}\xi_{h,d}(e)\cdot h|_{n+2} \sm \sqrt{e} & 0 \\ 0 & 1/\sqrt{e} \esm \Bigg | d | \tfrac{4}{l} \right \}.
$$
Here, $\CB(l)$ denotes the set of geometrically\footnote{with respect to the Petersson norm $\langle f,g \rangle = \int_{\lquotient{\Gamma_0(4)}{\BH}} f(z)\overline{g(z)}y^{n}dxdy$} normalised new forms $h$ of level $l$ with positive first Fourier coefficient $\widehat{h}(1)\in \BR^+$. Furthermore, $\xi_{h,d}(e)$ is some rather complicated arithmetic function, but we shall only require the bound
\begin{equation}
|\xi_{h,d}(e)| \ll_{\epsilon} d^{\epsilon}.
\label{eq:xibound}
\end{equation}
By Atkin--Lehner theory, each $\Phi_j$ corresponds to some newform $h$ of level $l|4$. Let $\vect{\lambda}_h$ denote the set of the odd Hecke eigenvalues of $h$. For each $j$ such that $\phi_j \in V_{\vect{\lambda}_h}$, we have an equality
\begin{equation}
\Phi_j(z) = \sum_{d| \frac{4}{l}} \eta_{j,d} h_d(z).
\label{eq:oldsplit}
\end{equation}
We are now able to compute $\|\tfrac{n+1}{8}F_n\|^2$ using its spectral expansion \eqref{eq:maineq}. Upon recalling $\vect{x}=\vect{y}$, we find it to equal
\begin{equation}\begin{aligned}
\sum_{l|4}\sum_{\substack{h \in \CB(l)}} \left \| \sum_{\phi_j \in V_{\vect{\lambda}_{h}}}  |\phi_j(\vect{x})|^2 \Phi_j \right \|^2 &=\sum_{l|4} \sum_{\substack{h \in \CB(l)}} \sum_{d | \frac{4}{l}} \left| \sum_{\phi_j \in V_{\vect{\lambda}_{h}}}  |\phi_j(\vect{x})|^2 \eta_{j,d} \right|^2 \\ 
&\ge \sum_{l|4} \sum_{\substack{h \in \CB(l)}} \frac{ \displaystyle \left|\sum_{\phi_j \in V_{\vect{\lambda}_{h}}}  |\phi_j(\vect{x})|^2 \sum_{d|\frac{4}{l}}\eta_{j,d} \xi_{h,d}(1) \right|^2}{\displaystyle\sum_{d|\frac{4}{l}} |\xi_{h,d}(1)|^2 } \\
&= \sum_{l|4} \sum_{\substack{h \in \CB(l)}} \frac{\displaystyle \left(\sum_{\phi_j \in V_{\vect{\lambda}_{h}}}  |\phi_j(\vect{x})|^2 \right)^2}{\displaystyle |\widehat{h}(1)|^2 \sum_{d|\frac{4}{l}} |\xi_{h,d}(1)|^2 } \\
& \gg_{\epsilon} (4\pi)^{-n} \Gamma(n+2) n^{-\epsilon} \sum_{\substack{\vect{\lambda}\\\lambda(1)=1}} \left( \sum_{\phi_j \in V_{\vect{\lambda}}}  |\phi_j(\vect{x})|^2 \right)^2.
\label{eq:eigenspacenorm}
\end{aligned}\end{equation}
In the above deduction, we have used Cauchy--Schwarz, the equality of the first Fourier coefficient in \eqref{eq:oldsplit}, i.e. $\widehat{h}(1)\sum_{d|\frac{4}{l}} \eta_{j,d}\xi_{h,d}(1)=1$, \eqref{eq:xibound}, and the Hoffstein--Lockhart \cite{HoffLock} upper bound on $|\widehat{h}(1)|^2$. It remains to bound the Petersson norm of $F_n$ geometrically. The following proposition suffices to conclude Theorem \ref{thm:22moment}.

\begin{prop} Let $n > 0 $ be an even integer. Then, we have
$$
\|F_n\|^2 
\ll_{\epsilon} (4 \pi)^{-n} \Gamma(n+2) \cdot n^{1+\epsilon}.
$$
\label{prop:Fnnorm}
\end{prop}
This estimate has already appeared in the author's thesis \cite{thesis} in a different context, but for the purpose of accessibility we shall reproduce the proof here. We shall require two proposisitons from the geometry of numbers.

\begin{prop}[Minkowski's second Theorem] Let $\CK\subseteq \BR^n$ be a closed convex $0$-symmetric set of positive volume. Let $\Lambda\subset \BR^n$ be a lattice and further let $ \lambda_1 \le \lambda_2 \le \dots \le \lambda_n$ be the successive minima of $\CK$ on $\Lambda$. Then, we have
$$
\frac{2^n}{n!} \vol(\BR^n/\Lambda) \le \lambda_1 \lambda_2 \cdots \lambda_n \vol(\CK) \le 2^n \vol(\BR^n/\Lambda).
$$
\label{prop:Mink2}
\end{prop}

\begin{prop} Let $\CK\subseteq \BR^n$ be a closed convex $0$-symmetric set of positive volume. Let $\Lambda\subset \BR^n$ be a lattice and further let $ \lambda_1 \le \lambda_2 \le \dots \le \lambda_n$ be the successive minima of $\CK$ on $\Lambda$. Then, we have
$$
|\CK \cap \Lambda| \le \prod_{i=1}^n \left( 1+\frac{2i}{\lambda_i} \right).
$$
\label{prop:latticepointcount}
\end{prop}
\begin{proof} See \cite[Prop. 2.1]{LattCount}.
\end{proof}

\begin{proof}[Proof of Proposition \ref{prop:Fnnorm}] We have
$$\begin{aligned}
\int_{\lquotient{\Gamma_0(4)}{\BH}} |F_n(z)|^2 y^{n+2} \frac{dxdy}{y^2}  &= 2^{-n-2} \int_{\lquotient{\Gamma_0(4)}{\BH}} |F_n(z)|^2 (2y)^{n+2} \frac{dxdy}{y^2} \\
 &= 2^{-n-2}  \int_{\lquotient{\Gamma(2)}{\BH}} |F_n(\tfrac{z}{2})|^2 y^{n+2} \frac{dxdy}{y^2} \\
&= 2^{-n-1}\int_{\lquotient{\Gamma_{\theta}}{\BH}} |G_n(z)|^2 y^{n+2} \frac{dxdy}{y^2}.
\end{aligned}$$
We further bound the latter integral by
$$
\int_{\frac{\sqrt{3}}{2}}^{\infty} \int_0^2 |G_n(z)|^2 y^{n} dxdy + \int_{\frac{\sqrt{3}}{2}}^{\infty} \int_0^1 \left|\left(G_n|_{n+2}\sm 1 & -1 \\ 1 & 0 \esm \right)(z)\right|^2 y^{n} dxdy = \CI_1+\CI_2, \text{ say}.
$$
We shall only deal with $\CI_1$ as the same lattice point counting argument may also be applied to $\CI_2$ due to the very nature of the Fourier expansion of $G_n$ at the cusp $1$, see \eqref{eq:cusp1}. We insert the Fourier expansion \eqref{eq:Fn} and integrate over $x$. We find
$$\begin{aligned}
\CI_1 
&= \int_{\frac{\sqrt{3}}{2}}^{\infty} \sum_{k=1}^{\infty} k^n e^{-2\pi k y} \left( \sum_{\substack{\vect{m} \in B(\BZ) \\ \nr(\vect{m})=k}} U_n \left( \frac{1}{2\sqrt{\nr(\vect{m})}} \tr \left ( \vect{m} \right ) \right) \right)^2 y^n dy \\
& \le \int_{\frac{\sqrt{3}}{2}}^{\infty} \sum_{k=1}^{\infty} k^n e^{-2\pi k y} \left( \sum_{\substack{\vect{m} \in B(\BZ) \\ \nr(\vect{m})=k}} \min \left\{ n+1, \frac{\sqrt{\nr(\vect{m})}}{\sqrt{m_2^2+m_3^2+m_4^2}} \right \} \right)^2 y^n dy,
\end{aligned}$$
where we have made use of the bound $U_n(x) \le \min \{ n+1, (1-x^2)^{-\frac{1}{2}}\}$, which is easy to read off the definition \eqref{eq:Un}. We shall first deal with the contribution from $k \ge 10 n$. In this case, we may bound the sum over $k$ (without the factor $y^n$) by
$$\begin{aligned}
\ll n^2 \sum_{k \ge 10 n}k^{n+3}e^{-2\pi ky} &\ll n^2 \sum_{k \ge 10n} n^{n+3} ( \pi y)^{-n-3} e^{-n} e^{-\pi k y} \\
&\ll n^{n+5}( \pi e)^{-n} y^{-n-3} e^{-10 \pi n y}.
\end{aligned}$$
Hence, the contribution from $k \ge 10 n$ towards $\CI_1$ is bounded by
$$
n^{n+5} ( \pi e)^{-n} \int_{\frac{\sqrt{3}}{2}}^{\infty} e^{-10 \pi n y} y^{-3} dy \ll n^{n+5} ( \pi e)^{-n} e^{-10n}.
$$
This is sufficient. For $k \le 10n$, we interchange the integral and summation in $\CI_1$. We further extend the integral all the way down to $0$ and find that the contribution is at most
\begin{multline}
(2 \pi)^{-n-1} \Gamma(n+1) \sum_{k=1}^{10n} \frac{1}{k}  \left( \sum_{\substack{\vect{m} \in B(\BZ) \\ \nr(\vect{m})=k}} \min \left\{ n+1, \frac{\sqrt{\nr(\vect{m})}}{\sqrt{m_2^2+m_3^2+m_4^2}} \right \} \right)^2 \\
=(2 \pi)^{-n-1} \Gamma(n+1) \left(\frac{1}{10n} A(10n)+\int_1^{10n} A(x) \frac{dx}{x^2} \right),
\label{eq:smallkintermed}
\end{multline}
where
$$
A(X) = \sum_{1\le k \le X} \left( \sum_{\substack{\vect{m} \in B(\BZ) \\ \nr(\vect{m})=k}} \min \left\{ n+1, \frac{\sqrt{\nr(\vect{m})}}{\sqrt{m_2^2+m_3^2+m_4^2}} \right \} \right)^2.
$$
In order to bound $A(X)$, we cover the quaternions $\vect{m}$ by sets $C(R)$ with $R=2^i$, $i \in \BN_0$. They are defined as follows
$$
\vect{m} \in C(R) \Leftrightarrow m_2^2+m_3^2+m_4^2  \le \frac{\nr(\vect{m})}{R^2}.
$$
Fix a $k$ and consider all points $\vect{m}\in C(R)$ with $\nr(\vect{m})=k$. We have $|m_1|=\sqrt{k}(1+O(R^{-2}))$. Thus, there are $\ll 1+k^{\frac{1}{2}}/R^2$ choices for $m_1$ and for any such choice of $m_1$ there are $\ll_{\epsilon} 1+k^{\frac{1}{2}+\epsilon}/R$ choices for $(m_2,m_3,m_4)$ satisfying $\nr(\vect{m})=k$. Hence, we deduce
\begin{equation}
\left|\left\{\vect{m}\in B(\BZ)| \nr(\vect{m})=k \text{ and } \vect{m}\in C(R)\right\}\right| \ll_{\epsilon} \left(1+\frac{k^{\frac{1}{2}}}{R}+\frac{k}{R^3} \right)k^{\epsilon}.
\label{eq:singlebound}
\end{equation}
We are now going to refine this estimate as $k$ varies in an interval $[M,2M]$. In this case, we have the conditions
$$
m_1^2 \le 2M  \text{ and } m_2^2+m_3^2+m_4^2  \le \frac{\nr(\vect{m})}{R^2} \le \frac{2M}{R^2}.
$$
This defines a $0$-symmetric cylinder $\CK$. By Proposition \ref{prop:latticepointcount}, the number of integral quaternions $\vect{m}$ inside $\CK$ is bounded by
$$
\ll \frac{1}{\lambda_1}+\frac{1}{\lambda_1\lambda_2}+\frac{1}{\lambda_1\lambda_2\lambda_3}+\frac{1}{\lambda_1\lambda_2\lambda_3\lambda_4}.
$$
Clearly, we have $\lambda_1 \gg M^{-\frac{1}{2}}$ and $\lambda_1\lambda_2\lambda_3\lambda_4 \gg M^{-2}R^{3}$ by Proposition \ref{prop:Mink2}. We also claim $\lambda_1\lambda_2\gg M^{-1}R$ and $\lambda_1\lambda_2\lambda_3 \gg M^{-\frac{3}{2}}R^{2}$. Let us illustrate this for $\lambda_1\lambda_2$. Let $\vect{v_1},\vect{v_2}$ be two linearly independent vectors in $B(\BZ)$ for which the second successive minima is attained. Then, $\BZ\vect{v_1}+\BZ\vect{v_2}$ is a lattice with co-volume at least $1$. Furthermore, we have $\vol(\CK \cap (\BR\vect{v_1}+\BR\vect{v_2})) \ll MR^{-1}$, which may be deduced from a general Pythagorean theorem \cite{GenPythagoras}. Hence, by Proposition \ref{prop:Mink2}, we have $\lambda_1\lambda_2 \gg M^{-1}R$. The bound $\lambda_1\lambda_2\lambda_3 \gg M^{-\frac{3}{2}}R^{2}$ follows from the same considerations. Thus, we find 
\begin{equation}
\left|\left\{\vect{m}\in B(\BZ)| M \le \nr(\vect{m})\le 2M \text{ and } \vect{m}\in C(R)\right\}\right| \ll M^{\frac{1}{2}}+\frac{M^2}{R^3}.
\label{eq:intbound}
\end{equation}
For our convenience, let us denote $D(R)=C(R)\backslash C(2R)$. From Cauchy--Schwarz, it follows that
\begin{multline}
A(2M)-A(M) = \sum_{M < k \le 2M} \left( \sum_{\substack{\vect{m} \in B(\BZ) \\ \nr(\vect{m})=k}} \min \left\{ n+1, \frac{\sqrt{\nr(\vect{m})}}{\sqrt{m_2^2+m_3^2+m_4^2}} \right \} \right)^2 \\
\ll \sum_{M < k \le 2M} \left( \sum_{i=0}^{\lfloor \log_2(n) \rfloor} \mu_i +\mu  \right) \left( \sum_{i=0}^{\lfloor \log_2(n) \rfloor} \frac{2^{2i}}{\mu_i} \left(\sum_{\substack{\nr(\vect{m})=k \\ \vect{m} \in D(2^i)}}1\right)^2+\frac{n^2}{\mu} \left(\sum_{\substack{\nr(\vect{m})=k \\ \vect{m} \in C(n)}}1\right)^2  \right)
\label{eq:Adyadic}
\end{multline}
for some positive weights $\mu_i,\mu$, which we shall choose in due time. Equations \eqref{eq:singlebound} and \eqref{eq:intbound} imply
$$
\sum_{M < k \le 2M} \left(\sum_{\substack{\nr(\vect{m})=k \\ \vect{m} \in D(R)}}1\right)^2 \ll_{\epsilon} \left ( M^{\frac{1}{2}}+ \frac{M^{\frac{5}{2}}}{R^4} + \frac{M^3}{R^6} \right) M^{\epsilon}.
$$
Hence, for $M \ll n$, \eqref{eq:Adyadic} is further bounded by
\begin{multline}
\ll_{\epsilon} \left( \sum_{i=0}^{\lfloor \log_2(n) \rfloor} \mu_i +\mu  \right) \left(\sum_{i=0}^{\lfloor \frac{1}{4}\log_2(M)\rfloor} \frac{2^{2i}}{\mu_i} \frac{M^{3+\epsilon}}{2^{6i}} + \sum_{i=\lfloor \frac{1}{4}\log_2(M)\rfloor+1}^{\lfloor \frac{1}{2} \log_2(M) \rfloor} \frac{2^{2i}}{\mu_i} \frac{M^{\frac{5}{2}+\epsilon}}{2^{4i}} \right) \\
+\left( \sum_{i=0}^{\lfloor \log_2(n) \rfloor} \mu_i +\mu  \right) \left ( \sum_{i=\lfloor \frac{1}{2}\log_2(M)\rfloor+1}^{\lfloor \log_2(n) \rfloor} \frac{2^{2i}}{\mu_i} M^{\frac{1}{2}+\epsilon} +  \frac{n^2}{\mu} M^{\frac{1}{2}+\epsilon} \right).
\label{eq:Adyadic2}
\end{multline}
We make the following choices for the weights: $\mu=n\cdot M^{\frac{1}{4}}$ and
$$
\mu_i = \begin{cases} M^{\frac{3}{2}}2^{-2i}, & 0\le i \le \lfloor\tfrac{1}{4} \log_2(M)\rfloor, \\
M^{\frac{5}{4}}2^{-i}, & \lfloor\tfrac{1}{4} \log_2(M)\rfloor < i \le \lfloor\tfrac{1}{2}\log_2(M)\rfloor, \\
M^{\frac{1}{4}} 2^i, & \lfloor\tfrac{1}{2}\log_2(M)\rfloor < i \le \lfloor \log_2(n) \rfloor. \end{cases}
$$
It follows that for $M \ll n$ we have
$$
A(2M)-A(M) \ll_{\epsilon} M^{3+\epsilon}+n^2M^{\frac{1}{2}+\epsilon}
$$
and hence $A(X) \ll_{\epsilon} X^{3+\epsilon}+n^2X^{\frac{1}{2}+\epsilon}$ for $X \ll n$ which when combined with \eqref{eq:smallkintermed} concludes the proposition.

\end{proof}

\bibliography{RafBib}
\end{document}